\documentclass[twoside,11pt,reqno]{amsart}
\usepackage{amsmath,amsthm,amssymb,amstext,amsfonts,amscd}
\usepackage{graphicx}
\usepackage{multirow}
\usepackage{lmodern}
\usepackage{latexsym}
\usepackage{wasysym}

\newtheorem{theorem}{Theorem}

\newtheorem{corollary}{Corollary}

\numberwithin{equation}{section}

\begin{document}
\title[{New mixed recurrence relations of two-variable orthogonal polynomials}]
{New mixed recurrence relations of two-variable orthogonal polynomials via differential operators}

\author[M.M. Makky and M. Shadab]{Mosaed M. Makky and Mohammad Shadab$^{*}$}

\address{Mosaed M. Makky: Departement of Mathematics, Faculty of Science, South Valley University, (Qena-Egypt).}
\email{mosaed\_makky11@yahoo.com, mosaed\_makky@sci.svu.edu.eg}

\address{Mohammad Shadab: Department of Natural and Applied Sciences,
 School of Science and Technology,
 Glocal University,
 Saharanpur 247121, India.}
\email{shadabmohd786@gmail.com}



\thanks{*Corresponding author}

\begin{abstract}
In this paper, we derive new recurrence relations for two-variable orthogonal polynomials for example Jacobi polynomial, Bateman's polynomial and Legendre polynomial via two different differential operators $\Xi =\left(\frac{\partial }{\partial z} +\sqrt{w} \frac{\partial
}{\partial w} \right)$ and $\Delta =\left(\frac{1}{w} \frac{\partial }{\partial z} +\frac{1}{z}
\frac{\partial }{\partial w} \right)$. We also derive some special cases of our main results.
\end{abstract}

\maketitle

\section{Introduction and preliminaries}
In recent decades, the study of the multi-variable orthogonal polynomials has been substantially developed by many authors \cite{2, Dunkl, 10}. The properties of the multi-variable orthogonal polynomials have analyzed by different approaches.The analytical properties of two-variable orthogonal polynomials like generating functions, recurrence relations, partial differential equations, and orthogonality have been remain the main attraction of the topic due to its wide range of applications in different research areas \cite{9, 11, 7, FSS, 13, 12}.\\

Some new classes of two-variables analogues of the Jacobi polynomials have been introduced from Jacobi weights by Koornwinder \cite{Koornwinder}. These all classes are introduced by means of two different partial differential operators $D_1$ and $D_2$, where $D_1$ has order two, and $D_2$ may have any arbitrary order. Koornwinder constructed bases of orthogonal polynomials in two-variables by using a tool given by Agahanov \cite{Agahanov}.\\

In 2017, M. Marriaga et al. \cite{Marriaga} derived some new recurrence relations involving two-variable orthogonal polynomials in a different way. In 2019, G.V. Milovanovic et at. \cite{Milovanovic} presented the study of various recurrence relations, generating functions and series expansion formulas for two families of orthogonal polynomials in two-variables. Motivated by these two studies, we present here some recurrence relations of two-variables orthogonal polynomials via differential operators.\\

The generalized hypergeometric function \cite[p.42-43]{Sri-Man} can be defined as
\begin{equation}\label{f.eq(g2)}
{_pF_q}\left[\begin{array}{c}\alpha_1,\ldots,\alpha_p;\\\beta_1,\ldots,\beta_q;\end{array}z\right]=\sum_{n=0}^{\infty}\frac{(\alpha_1)_n\ldots(\alpha_p)_n}{(\beta_1)_n\ldots(\beta_q)_n}\frac{z^n}{n!},
\end{equation}
with certain convergence conditions given in \cite[p.43]{Sri-Man}.\\

The Pochhammer symbol $(\lambda)_{\nu}$ ~$(\lambda, \nu \in\mathbb{C})$ \cite[p.22 eq(1)]{Rainville}, is defined by
\begin{equation}\label{f.eq(g1)}
\left(\lambda\right)_{\nu}:=\frac{\Gamma\left(\lambda+\nu\right)}{\Gamma\left(\lambda\right)}=\begin{cases}
\begin{array}{c}
1\\
~\\
\lambda\left(\lambda+1\right)\ldots\left(\lambda+n-1\right)
\end{array} & \begin{array}{c}
\left(\nu=0;\lambda\in\mathbb{C}\setminus\left\{ 0\right\} \right)\\
~\\
\left(\nu=n\in\mathbb{N};\lambda\in\mathbb{C}\right),
\end{array}\end{cases}
\end{equation}
it being understood $conventionally$ that $\left(0\right)_{0}=1$, and assumed $tacitly$ that the $\Gamma$ quotient exists.\\

The classical Jacobi polynomial $P_{n}^{(\alpha ,\beta )} (x)$ of degree $n$ $(n=0,1,2,\dots)$ \cite[p. 254(1)]{Rainville} defined as
\begin{eqnarray} \label{GrindEQ__1_3_}
&&P_{n}^{(\alpha ,\beta )} (x)=\frac{(1+\alpha )_{n} }{n!} {}_{2}
F_{1} \left(-n,1+\alpha +\beta +n;1+\alpha ;\frac{1-x}{2} \right),\\
&&\hskip10mm\Re(\alpha) >-1,\; \; \Re(\beta) >-1, x\in (-1,1).\nonumber
\end{eqnarray}

The generating function of the Jacobi polynomial $P_{n}^{(\alpha ,\beta )} (x)$ of degree $n$ \cite[p. 270(2)]{Rainville} is defined by
\begin{equation} \label{GrindEQ__1_4A_}
\sum _{n=0}^{\infty } P_{n}^{(\alpha ,\beta )} (x)\; t^{n} =F_{4} \left(1+\beta ,1+\alpha ;1+\alpha ,1+\beta ;\frac{1}{2} t(x-1),\frac{1}{2} t(x+1)\right)\end{equation}
where\\
$$F_{4} \left(1+\beta ,1+\alpha ;1+\alpha ,1+\beta ;\frac{1}{2} t(x-1),\frac{1}{2} t(x+1)\right),$$ is Appell polynomial \cite[p. 53 Eq. (7)]{Sri-Man}.\\

\noindent An elementary generating function of the Jacobi polynomial $P_{n}^{(\alpha ,\beta )} (x)$ \cite[p. 271 Eq. (6)]{Rainville} can presented in the form
\begin{equation} \label{GrindEQ__1_4_}
\sum _{n=0}^{\infty } P_{n}^{(\alpha ,\beta )} (x)t^{n}
= \rho^{-1}
\left(\frac{2}{1+t+\rho } \right)^{\beta }
\left(\frac{2}{1-t+\rho } \right)^{\alpha },
\end{equation}
or
\begin{equation}
\sum _{n=0}^{\infty } P_{n}^{(\alpha ,\beta )} (x)t^{n} =2^{\alpha +\beta } \rho^{-1} \left(1+t+\rho  \right)^{-\beta } \left(1-t+\rho  \right)^{-\alpha } ,
\end{equation}
where, $\rho=\left(1-2xt+t^{2} \right)^{\frac{1}{2} }$, and on setting $\alpha =\beta =0$, the Jacobi polynomial reduce to the Legendre Polynomial.\\

Recently, R. Khan et al. \cite{5} introduced generalization of two-variable Jacobi polynomial
\begin{eqnarray} \label{GrindEQ__1_8_}
&&P_{n}^{(\alpha ,\beta )} (x,y)=\sum _{k=0}^{n} \frac{\left(1+\alpha
\right)_{n} \left(1+\alpha +\beta \right)_{n+k}
}{k!(n-k)!\left(1+\alpha \right)_{k} \left(1+\alpha +\beta
\right)_{n} } \left(\frac{x-\sqrt{y} }{2} \right)^{k},\\
&&\hskip10mm \left(n=0,1,2,\dots; \Re(\alpha) >-1,\; \; \Re(\beta) >-1, x,y \in (-1,1)\right)\nonumber
\end{eqnarray}

which can be presented in the alternate form

\begin{equation} \label{GrindEQ__1_6_}
P_{n}^{(\alpha ,\beta )} (x,y)=\sum _{n,k=0}^{\infty }
\frac{\left(1+\alpha \right)_{n} \left(1+\beta \right)_{n}
}{k!(n-k)!\left(1+\alpha \right)_{k} \left(1+\beta \right)_{n-k} }
\left(\frac{x-\sqrt{y} }{2} \right)^{k} \left(\frac{x+\sqrt{y} }{2}
\right)^{n-k}
\end{equation}
and
\begin{equation} \label{GrindEQ__1_7_}
P_{n}^{(\alpha ,\beta )} (x,y)=\frac{\left(1+\alpha \right)_{n}
}{n!} \left(\frac{x+\sqrt{y} }{2} \right)^{n} {}_{2} F_{1}
\left(-n,-\beta -n;1+\alpha ;\frac{x-\sqrt{y} }{x+\sqrt{y} } \right)
\end{equation}
or
\begin{equation}P_{n}^{(\alpha ,\beta )} (x,y)=\frac{\left(1+\alpha \right)_{n} }{n!} {}_{2} F_{1} \left(-n,1+\alpha +\beta +n;1+\alpha ;\frac{\sqrt{y} -x}{2} \right).
\end{equation}

The generating functions of generalized Jacobi polynomial of two-variables $P_{n}^{(\alpha
,\beta )} (x,y)$ \cite{5} can be presented as follows
\begin{equation}\begin{array}{l} {\sum _{n=0}^{\infty } P_{n}^{(\alpha ,\beta )} (x,y)\; t^{n} } = \mu^{-1} \left(\frac{2}{1+\sqrt{y} t+\mu } \right)^{\beta } \left(\frac{2}{1-\sqrt{y} t+\mu } \right)^{\alpha } \end{array},\end{equation}
or
\begin{equation}\begin{array}{l} {\sum _{n=0}^{\infty } P_{n}^{(\alpha ,\beta )} (x,y)t^{n} } = 2^{\alpha +\beta } \mu^{-1} \left(1+\sqrt{y} t+\mu \right)^{-\beta } \left(1-\sqrt{y} t+\mu \right)^{-\alpha } \end{array},\end{equation}
\\

where, $\mu=\left(1-2xt+y\; t^{2} \right)^{\frac{1}{2} }.$\\

In another way, the generating function of generalized Jacobi polynomials of two variables $P_{n}^{(\alpha
,\beta )} (x,y)$ \cite{5} can be presented as follows
\begin{equation} \label{GrindEQ__1_5_}
\sum _{n=0}^{\infty } P_{n}^{(\alpha ,\beta )} (x,y)t^{n} =F_{4}
\left(1+\beta ,1+\alpha ;1+\alpha ,1+\beta ;\frac{1}{2} t(x-\sqrt{y}
),\frac{1}{2} t(x+\sqrt{y} )\right),
\end{equation}
which can be written in the form
\begin{equation}\sum _{n=0}^{\infty } P_{n}^{(\alpha ,\beta )} (x,y)\; t^{n} =\sum _{n,k=0}^{\infty } \frac{\left(1+\alpha \right)_{n+k} \left(1+\beta \right)_{n+k} \frac{1}{2} \left(x-\sqrt{y} \right)^{k} \frac{1}{2} \left(x+\sqrt{y} \right)^{n} t^{n} }{k!n!\left(1+\alpha \right)_{k} \left(1+\beta \right)_{n} }.
\end{equation}

Bateman's polynomial, and its generating function \cite{5} can be deduce from equation \eqref{GrindEQ__1_8_} as follows

\begin{eqnarray} \label{GrindEQ__3_1_}
&&{\rm B} _{n}^{(\alpha ,\beta )} (x,y)=\left[\sum _{n=0}^{\infty }
\frac{\frac{1}{2} \left(x-\sqrt{y} \right)^{n} t^{n}
}{n!\left(1+\alpha \right)_{n} } \right]\left[\sum _{n=0}^{\infty }
\frac{\frac{1}{2} \left(x+\sqrt{y} \right)^{n} t^{n}
}{n!\left(1+\beta \right)_{n} } \right]\\
&&\hskip20mm \left(\Re(\alpha) >-1, \; \Re(\beta) >-1, |x|<1, |y|<1 \right). \nonumber
\end{eqnarray}
and
\begin{equation}
{\rm B} _{n}^{(\alpha ,\beta )} (x,y)=\sum _{n=0}^{\infty } \frac{P_{n}^{(\alpha ,\beta )} (x,y)t^{n} }{\left(1+\alpha \right)_{n} \left(1+\beta \right)_{n} } .
\end{equation}

The generalized Jacobi polynomial of two-variables $P_{n}^{(\alpha ,\beta )} (x,y)$ reduces to the Legendre polynomial of two variables $P_{n} (x,y)$ for $\alpha =\beta =0$ in \eqref{GrindEQ__1_8_}

\begin{equation} \label{GrindEQ__1_9_}
P_{n} (x,y)=\sum _{k=0}^{n} \frac{(n+k)!}{\left(k!\right)^{2} (n-k)!} \left(\frac{x-\sqrt{y} }{2} \right)^{k},
\end{equation}

and its generating function can be given by

\begin{equation}
\sum _{n=0}^{\infty } P_{n} (x,y)\; t^{n} =\left(1-2xt+yt^{2} \right)^{-\frac{1}{2} } .
\end{equation}

Also, Khan and Abukhammash \cite{Kha-Abu} defined the Legendre Polynomials of
two-variables $P_{n} (x,y)$ as
\begin{equation}
P_{n} (x,y)=\sum _{k=0}^{[{n\mathord{\left/ {\vphantom {n 2}} \right. \kern-\nulldelimiterspace} 2} ]} \frac{\left(-y\right)^{k} \left(\frac{1}{2} \right)_{n-k} \left(2x\right)^{n-2k} }{k!(n-k)!}
\end{equation}

and the generating function for $P_{n} (x,y)$ is given by

\begin{equation}
\sum _{k=0}^{n} P_{n} (x,y)\; t^{n} =\left(1-2xt+y\; t^{2} \right)^{\frac{1}{2} } \end{equation}.

\section {\bf  Recurrence relations for Jacobi polynomials}

In this section, we will study the action of the following  differential operator
\begin{equation} \label{GrindEQ__2_1_}
\Xi =\left(\frac{\partial }{\partial z} +\sqrt{w} \frac{\partial
}{\partial w} \right)
\end{equation}
on complex bivariate Jacobi polynomial $P_{n}^{(\alpha ,\beta )}(z,w)$ \eqref{GrindEQ__2_2_} to obtain the desired results.\\

Now, we present complex bivariate Jacobi polynomial by replacing $x,y \in \mathbb{R}$ by $z,w \in \mathbb{C}$ such that
\begin{eqnarray} \label{GrindEQ__2_2_}
&&P_{n}^{(\alpha ,\beta )} (z,w)=\sum _{k=0}^{n} \frac{\left(1+\alpha
\right)_{n} \left(1+\alpha +\beta \right)_{n+k}
}{k!(n-k)!\left(1+\alpha \right)_{k} \left(1+\alpha +\beta
\right)_{n} } \left(\frac{z-\sqrt{w} }{2} \right)^{k}\\
&&\hskip20mm \left(\Re(\alpha) >-1, \; \Re(\beta) >-1, |z|<1, |w|<1 \right)\nonumber.
\end{eqnarray}

Following conjugate relations will be use frequently in the paper.
\begin{eqnarray}
&&\left(1+\alpha +\beta \right)_{n+k+1} =\left(1+\alpha +\beta \right)\left(2+\alpha +\beta \right)\left[1+(1+\alpha )+(1+\beta )\right]_{(n-1)+k};\\
&&\left(1+\alpha \right)_{k+1} =\left(1+\alpha \right)\left(1+(1+\alpha )\right)_{k}; \\
&&\left(1+\alpha \right)_{n} =\left(1+\alpha \right)\left(1+(1+\alpha )\right)_{n-1}; \\
&&\left(1+\alpha +\beta \right)_{n+1} =\left(1+\alpha +\beta \right)_{n} \left(1+\alpha +\beta +n\right).
\end{eqnarray}

\begin{theorem}
 Following recurrence relation for the Jacobi Polynomial $P_{n}^{(\alpha ,\beta )} (z,w)$ holds true
\begin{eqnarray} \label{GrindEQ__2_3_}
&&\frac{\partial }{\partial z} \; P_{n}^{(\alpha ,\beta )} (z,w)+ \sqrt{w} \frac{\partial }{\partial w}\; P_{n}^{(\alpha ,\beta )} (z,w)-\frac{\left(1+\alpha +\beta
+n\right)}{4} \; P_{n-1}^{(1+\alpha ),(1+\beta )} (z,w)=0\nonumber\\
\\
&&\hskip20mm \left(\Re(\alpha) >-1, \; \Re(\beta) >-1, |z|<1, |w|<1 \right)\nonumber.
\end{eqnarray}
\end{theorem}

\begin{proof}
On applying the operator \eqref{GrindEQ__2_1_} in \eqref{GrindEQ__2_2_}, we get

\begin{eqnarray}
&& \left(\frac{\partial }{\partial z} +\sqrt{w} \frac{\partial }{\partial w} \right)P_{n}^{(\alpha ,\beta )} (z,w)\nonumber\\
&&=\left(\frac{\partial }{\partial z} +\sqrt{w} \frac{\partial }{\partial w} \right)\sum _{k=0}^{n} \frac{\left(1+\alpha \right)_{n} \left(1+\alpha +\beta \right)_{n+k} }{k!(n-k)!\left(1+\alpha \right)_{k} \left(1+\alpha +\beta \right)_{n} } \left(\frac{z-\sqrt{w} }{2} \right)^{k}\nonumber\\
&&=\sum _{k=0}^{n} \frac{k\left(1+\alpha \right)_{n} \left(1+\alpha +\beta \right)_{n+k} }{k!(n-k)!\left(1+\alpha \right)_{k} \left(1+\alpha +\beta \right)_{n} } \left(\frac{z-\sqrt{w} }{2} \right)^{k-1} \left(\frac{1}{2} -\frac{1}{4} \right)\nonumber\\
&&= \frac{1}{4} \sum_{k=0}^{n} \frac{k\left(1+\alpha \right)_{n} \left(1+\alpha +\beta\right)_{n+k}}{k!(n-k)!\left(1+\alpha \right)_{k} \left(1+\alpha +\beta \right)_{n} } \left(\frac{z-\sqrt{w} }{2} \right)^{k-1}.\nonumber
\end{eqnarray}
Now, on replacing $k\rightarrow k+1$ and simplifications, we get
\begin{eqnarray}
&&=\frac{1}{4} \sum _{k=0}^{n} \frac{\left(1+\alpha \right)_{n} \left(1+\alpha +\beta \right)_{n+k+1} }{k!\left[n-(k+1)\right]!\left(1+\alpha \right)_{k+1} \left(1+\alpha +\beta \right)_{n} } \left(\frac{z-\sqrt{w} }{2} \right)^{k}\nonumber\\
&&=\frac{1}{4} \sum _{k=0}^{n} \frac{\left(1+\alpha \right)\left(1+(1+\alpha )\right)_{n-1} \left(1+\alpha +\beta \right)\left(2+\alpha +\beta \right) }{k!\left((n-1)-k\right)!\left(1+\alpha \right)\left(1+(1+\alpha )\right)_{k} }\nonumber\\
&&\times\frac{\left(1+(1+\alpha )+(1+\beta )\right)_{(n-1)+k}\left(1+\alpha +\beta +n\right)}{\left(1+(1+\alpha )+(1+\beta )\right)_{(n-1)} \left(1+\alpha +\beta \right)\left(2+\alpha +\beta \right)} \left(\frac{z-\sqrt{w} }{2} \right)^{k}\nonumber\\
&&=\frac{1+\left(\alpha +\beta +n\right)}{4} \sum _{k=0}^{n} \frac{\left(1+(1+\alpha )\right)_{n-1} \left[1+(1+\alpha )+(1+\beta )\right]_{(n-1)+k} }{k!\left((n-1)-k\right)!\left(1+(1+\alpha )\right)_{k} \left[1+(1+\alpha )+(1+\beta )\right]_{n-1} } \; \left(\frac{z-\sqrt{w} }{2} \right)^{k}\nonumber\\
&&=\frac{(1+\left(\alpha +\beta +n\right)}{4} \; P_{n-1}^{(1+\alpha ),(1+\beta )} (z,w).
\end{eqnarray}

Therefore, we get the desired result.\\
\end{proof}

\begin{corollary}
 Following recurrence relation  for the Jacobi Polynomial $P_{n}^{(\alpha ,\beta )} (z,1)$ holds true
\begin{eqnarray} \label{GrindEQ__2_3_}
&&\frac{\partial }{\partial z} \; P_{n}^{(\alpha ,\beta )} (z,1)-\frac{(1+\left(\alpha +\beta
+n\right)}{2} \; P_{n-1}^{(1+\alpha ),(1+\beta )} (z,1)=0\\
&&\hskip20mm \left(\Re(\alpha) >-1, \; \Re(\beta) >-1, |z|<1 \right)\nonumber.
\end{eqnarray}
\end{corollary}

\begin{proof}
First, put w=1 in \eqref{GrindEQ__2_2_} we consider the
Jacobi polynomials
\begin{equation} \label{GrindEQ__2_4_}
P_{n}^{(\alpha ,\beta )} (z,1)=\sum _{k=0}^{n} \frac{\left(1+\alpha
\right)_{n} \left(1+\alpha +\beta \right)_{n+k}
}{k!(n-k)!\left(1+\alpha \right)_{k} \left(1+\alpha +\beta
\right)_{n} } \left(\frac{z-1}{2} \right)^{k}.
\end{equation}

Taking differential operator $\Xi _{z} =\left(\frac{\partial}{\partial z} \right)$ then following the same process used in the above theorem leads to the desired result.
\end{proof}

\begin{corollary}
Following recurrence relation for the Jacobi Polynomial $P_{n}^{(\alpha ,\beta )} (1,w)$ holds true
\begin{eqnarray}
&&\left(\sqrt{w} \frac{\partial }{\partial w} \right)\,P_{n}^{(\alpha ,\beta )} (1,w)+\frac{(1+\left(\alpha +\beta +n\right)}{4} \; P_{n-1}^{(1+\alpha ),(1+\beta )} (1,w)=0\nonumber\\
\\
&&\hskip20mm \left(\Re(\alpha) >-1, \; \Re(\beta) >-1, |w|<1 \right)\nonumber.
\end{eqnarray}
\end{corollary}

\begin{proof}
 Put z=1 in \eqref{GrindEQ__2_2_} now we get
\[P_{n}^{(\alpha ,\beta )} (1,w)=\sum _{k=0}^{n} \frac{\left(1+\alpha \right)_{n} \left(1+\alpha +\beta \right)_{n+k} }{k!(n-k)!\left(1+\alpha \right)_{k} \left(1+\alpha +\beta \right)_{n} } \left(\frac{1-\sqrt{w} }{2} \right)^{k} \]

Taking differential operator $\Xi _{w} =\left(\sqrt{w} \frac{\partial }{\partial w} \right)$ then following the same process used in the above theorem leads to the desired result.
\end{proof}

\section {\bf  Recurrence relations for Bateman's polynomials}

\noindent

\noindent Now, we present complex bivariate Bateman's polynomial by replacing $x,y \in \mathbb{R}$ by $z,w \in \mathbb{C}$ such that

\begin{eqnarray} \label{GrindEQ__3_1_}
{\rm B} _{n}^{(\alpha ,\beta )} (z,w)=\left[\sum _{n=0}^{\infty }
\frac{\frac{1}{2} \left(z-\sqrt{w} \right)^{n} t^{n}
}{n!\left(1+\alpha \right)_{n} } \right]\left[\sum _{n=0}^{\infty }
\frac{\frac{1}{2} \left(z+\sqrt{w} \right)^{n} t^{n}
}{n!\left(1+\beta \right)_{n} } \right]
\end{eqnarray}
and
\begin{eqnarray}
&&\hskip5mm{\rm B} _{n}^{(\alpha ,\beta )} (z,w)=\sum _{n=0}^{\infty } \frac{P_{n}^{(\alpha ,\beta )} (z,w)t^{n} }{\left(1+\alpha \right)_{n} \left(1+\beta \right)_{n} }\\
&& \left(\Re(\alpha) >-1, \; \Re(\beta) >-1, |z|<1, |w|<1 \right) \nonumber.
\end{eqnarray}

We can also write the conjugate relationships for the purpose to use these relations in this section.
\begin{eqnarray}
&&\left(1+\alpha \right)_{n+1} =\left(1+\alpha \right)\left(1+(1+\alpha )\right)_{n}; \\
&&\left(1+\beta \right)_{n+1} =\left(1+\beta \right)\left(1+(1+\beta )\right)_{n} .
\end{eqnarray}

\begin{theorem}
Following recurrence relation for the Bateman's polynomial ${\rm B} _{n}^{(\alpha ,\beta )} (z,w)$ holds true
\\
\begin{eqnarray} \label{GrindEQ__3_2_}
&&\frac{\partial }{\partial z} \; {\rm B} _{n}^{(\alpha ,\beta )}
(z,w)+ \sqrt{w} \frac{\partial }{\partial w}\; {\rm B} _{n}^{(\alpha ,\beta )}(z,w)-\frac{\;t}{2\left(1+\alpha \right)} {\rm B} _{n}^{\left[(1+\alpha) ,\beta
\right]} (z,w)-\frac{3\;t}{2\left(1+\beta \right)} {\rm B} _{n}^{\left[\alpha ,(1+\beta
)\right]} (z,w)=0\nonumber\\
\\
&&\hskip20mm \left(\Re(\alpha) >-1, \; \Re(\beta) >-1, |z|<1, |w|<1 \right)\nonumber.
\end{eqnarray}
\end{theorem}

\begin{proof} Using the differential operator $\Xi \; $ for the
Bateman's polynomial of two variables ${\rm B}
_{n}^{(\alpha ,\beta )} (z,w)$ we see that
\[\Xi \; {\rm B} _{n}^{(\alpha ,\beta )} (z,w)\; =\left(\frac{\partial }{\partial z} +\sqrt{w} \frac{\partial }{\partial w} \right)\left[\sum _{n=0}^{\infty } \frac{\frac{1}{2} \left(z-\sqrt{w} \right)^{n} t^{n} }{n!\left(1+\alpha \right)_{n} } \right]\left[\sum _{n=0}^{\infty } \frac{\frac{1}{2} \left(z+\sqrt{w} \right)^{n} t^{n} }{n!\left(1+\beta \right)_{n} } \right] \]
\\
\[=\left[\sum _{n=0}^{\infty } \frac{\frac{n}{2} \left(z-\sqrt{w} \right)^{n-1} t^{n} }{n!\left(1+\alpha \right)_{n} } \right]\left[\sum _{n=0}^{\infty } \frac{\frac{1}{2} \left(z+\sqrt{w} \right)^{n} t^{n} }{n!\left(1+\beta \right)_{n} } \right]+\left[\sum _{n=0}^{\infty } \frac{\frac{1}{2} \left(z-\sqrt{w} \right)^{n} t^{n} }{n!\left(1+\alpha \right)_{n} } \right]\left[\sum _{n=0}^{\infty } \frac{\frac{n}{2} \left(z+\sqrt{w} \right)^{n-1} t^{n} }{n!\left(1+\beta \right)_{n} } \right]\]
\\
\[- \left[\sum _{n=0}^{\infty } \frac{\frac{n}{4} \left(z-\sqrt{w} \right)^{n-1} t^{n} }{n!\left(1+\alpha \right)_{n} } \right]\left[\sum _{n=0}^{\infty } \frac{\frac{1}{2} \left(z+\sqrt{w} \right)^{n} t^{n} }{n!\left(1+\beta \right)_{n} } \right]+\left[\sum _{n=0}^{\infty } \frac{\frac{1}{2} \left(z-\sqrt{w} \right)^{n} t^{n} }{n!\left(1+\alpha \right)_{n} } \right]\left[\sum _{n=0}^{\infty } \frac{\frac{n}{4} \left(z+\sqrt{w} \right)^{n-1} t^{n} }{n!\left(1+\beta \right)_{n} } \right] \]
\\
\[{=\frac{t}{2\left(1+\alpha \right)} \left[\sum _{n=0}^{\infty } \frac{\frac{1}{2} \left(z-\sqrt{w} \right)^{n} t^{n} }{n!\left[1+\left(1+\alpha \right)\right]_{n} } \right]\left[\sum _{n=0}^{\infty } \frac{\frac{1}{2} \left(z+\sqrt{w} \right)^{n} t^{n} }{n!\left(1+\beta \right)_{n} } \right]}\]
\[{\; \; +\frac{3\; t}{2\left(1+\beta \right)} \left[\sum _{n=0}^{\infty } \frac{\frac{1}{2} \left(z-\sqrt{w} \right)^{n} t^{n} }{n!\left(1+\alpha \right)_{n} } \right]\left[\sum _{n=0}^{\infty } \frac{\frac{1}{2} \left(z+\sqrt{w} \right)^{n} t^{n} }{n!\left[1+\left(1+\beta \right)\right]_{n} } \right]}\]

\[=\frac{t}{2\left(1+\alpha \right)} {\rm B} _{n}^{\left[(1+\alpha ),\beta \right]} (z,w)-\frac{3\; t}{2\left(1+\beta \right)} {\rm B} _{n}^{\left[\alpha ,(1+\beta )\right]} (z,w).\]
Therefore, we get the desired result.
\end{proof}

\begin{corollary}
 Following recurrence relation for the Bateman's polynomial ${\rm B} _{n}^{(\alpha ,\beta )} (z,1)$ holds true
\begin{eqnarray} \label{GrindEQ__3_3_}
&&\frac{\partial }{\partial z} \; {\rm B} _{n}^{(\alpha ,\beta )} (z,1)-\frac{t}{(1+\alpha )} {\rm B} _{n}^{\left[(1+\alpha ),\beta \right]} (z,1)-\frac{t}{(1+\beta )} {\rm B} _{n}^{\left[\alpha ,(1+\beta )\right]} (z,1)=0\\
&&\hskip20mm \left(\Re(\alpha) >-1, \; \Re(\beta) >-1, |z|<1 \right)\nonumber.
\end{eqnarray}
\end{corollary}

\begin{proof}
First, substitute w=1 in the Bateman's polynomial
\eqref{GrindEQ__3_1_}, we have
\begin{equation} \label{GrindEQ__3_3_}
{\rm B} _{n}^{(\alpha ,\beta )} (z,1)=\sum _{n=0}^{\infty }
\frac{\frac{1}{2} \left(z-1\right)^{n} t^{n} }{n!\left(1+\alpha
\right)_{n} } \sum _{n=0}^{\infty } \frac{\frac{1}{2}
\left(z+1\right)^{n} t^{n} }{n!\left(1+\beta \right)_{n} }
\end{equation}

Taking differential operator $\Xi _{z} =\left(\frac{\partial}{\partial z} \right)$ then following the same process used in the above theorem leads to the desired result.
\end{proof}

\begin{corollary}
 Following recurrence relation for the Bateman's polynomial ${\rm B} _{n}^{(\alpha ,\beta )} (1,w)$ holds true
\begin{eqnarray} \label{GrindEQ__3_4_}
&&\sqrt{w} \frac{\partial }{\partial w}  {\rm B} _{n}^{(\alpha ,\beta )} (1,w)+\frac{t}{2(1+\alpha
)} {\rm B} _{n}^{\left[(1+\alpha ),\beta \right]}
(1,w)-\frac{t}{2(1+\beta )} {\rm B} _{n}^{\left[\alpha ,(1+\beta
)\right]} (1,w)=0\nonumber\\
\\
&&\hskip20mm \left(\Re(\alpha) >-1, \; \Re(\beta) >-1, |w|<1 \right)\nonumber.
\end{eqnarray}
\end{corollary}

\begin{proof}
 Put z=1 in the Bateman's polynomial
\eqref{GrindEQ__3_1_}, we get
\[{\rm B} _{n}^{(\alpha ,\beta )} (1,w)=\left[\sum _{n=0}^{\infty } \frac{\frac{1}{2} \left(1-\sqrt{w} \right)^{n} t^{n} }{n!\left(1+\alpha \right)_{n} } \right]\left[\sum _{n=0}^{\infty } \frac{\frac{1}{2} \left(1+\sqrt{w} \right)^{n} t^{n} }{n!\left(1+\beta \right)_{n} } \right].\]

Taking differential operator $\Xi _{w} =\left(\sqrt{w} \frac{\partial }{\partial w} \right)$ then following the same process used in the above theorem leads to the desired result.
\end{proof}

\section{Recurrence relations for Legendre polynomials}

In this sections, we will study the action of the following  differential operator
\begin{equation} \label{GrindEQ__4_1_}
\Delta =\left(\frac{1}{w} \frac{\partial }{\partial z} +\frac{1}{z}
\frac{\partial }{\partial w} \right),
\end{equation}
on complex bivariate Legendre polynomial $P_{n} (z,w)$ \eqref{GrindEQ__2_2_} to obtain the desired results.\\

Now, we present complex bivariate Legendre polynomial by replacing $x,y \in \mathbb{R}$ by $z,w \in \mathbb{C}$ such that
\begin{equation} \label{GrindEQ__4_2_}
P_{n} (z,w)=\sum _{k=0}^{[{n\mathord{\left/ {\vphantom {n 2}}
\right. \kern-\nulldelimiterspace} 2} ]} \frac{\left(-w\right)^{k}
\left(\frac{1}{2} \right)_{n-k} \left(2z\right)^{n-2k} }{k!(n-k)!}
\end{equation}

where, $\Re(\alpha) >-1, \; \Re(\beta) >-1$, $|z|<1, |w|<1.$\\

\begin{theorem}
Following recurrence relation for the Legendre polynomials $P_{n} (z,w)$ holds true
\begin{eqnarray} \label{GrindEQ__4_3_}
&&\frac{1}{w} \frac{\partial }{\partial z}P_{n} (z,w)\
+\left(\frac{1}{z} \frac{\partial }{\partial w}\right)P_{n} (z,w)\
-\left(\frac{n}{zw} \right)P_{n} (z,w)\;
+\left(\frac{1}{2z^{2} } \right)P_{n-1} (z,w)=0\nonumber\\
\\
&&\hskip20mm \left(\Re(\alpha) >-1, \; \Re(\beta) >-1, |z|<1, |w|<1 \right)\nonumber.
\end{eqnarray}
\end{theorem}

\begin{proof} For Legendre polynomials \eqref{GrindEQ__4_2_} of two
variables $P_{n} (z,w)$ we see that

\[{\Delta \; P_{n} (z,w)=\left(\frac{1}{w} \frac{\partial }{\partial z} +\frac{1}{z} \frac{\partial }{\partial w} \right)\sum _{k=0}^{[{n\mathord{\left/ {\vphantom {n 2}} \right. \kern-\nulldelimiterspace} 2} ]} \frac{\left(-w\right)^{k} \left(\frac{1}{2} \right)_{n-k} \left(2z\right)^{n-2k} }{k!(n-k)!} }\]

\[{=-2\sum _{k=0}^{[{n\mathord{\left/ {\vphantom {n 2}} \right. \kern-\nulldelimiterspace} 2} ]} \frac{\left(-w\right)^{k-1} \left(\frac{1}{2} \right)_{n-k} \left(n-2k\right)\left(2z\right)^{n-2k-1} }{k!(n-k)!} -2\sum _{k=0}^{[{n\mathord{\left/ {\vphantom {n 2}} \right. \kern-\nulldelimiterspace} 2} ]} \frac{k\left(-w\right)^{k-1} \left(\frac{1}{2} \right)_{n-k} \left(2z\right)^{n-2k-1} }{k!(n-k)!} } \]

 \[\; \; \; \; \; =-2\sum _{k=0}^{[{n\mathord{\left/ {\vphantom {n 2}} \right. \kern-\nulldelimiterspace} 2} ]} \frac{\left(n-k\right)\left(-w\right)^{k-1} \left(\frac{1}{2} \right)_{n-k} \left(2z\right)^{n-2k-1} }{k!(n-k)!} \]
\[\; \; \; \; \; =-2n\sum _{k=0}^{[{n\mathord{\left/ {\vphantom {n 2}} \right. \kern-\nulldelimiterspace} 2} ]} \frac{\left(-w\right)^{k-1} \left(\frac{1}{2} \right)_{n-k} \left(2z\right)^{n-2k-1} }{k!(n-k)!} \; -2\sum _{k=0}^{[{n\mathord{\left/ {\vphantom {n 2}} \right. \kern-\nulldelimiterspace} 2} ]} \frac{k\left(-w\right)^{k-1} \left(\frac{1}{2} \right)_{n-k} \left(2z\right)^{n-2k-1} }{k!(n-k)!} \]

\[=\left(\frac{n}{zw} \right)\sum _{k=0}^{[{n\mathord{\left/ {\vphantom {n 2}} \right. \kern-\nulldelimiterspace} 2} ]} \frac{\left(-w\right)^{k} \left(\frac{1}{2} \right)_{n-k} \left(2z\right)^{n-2k} }{k!(n-k)!} \; -2\left(\frac{1}{2z} \right)^{2} \sum _{k=0}^{[{n\mathord{\left/ {\vphantom {n 2}} \right. \kern-\nulldelimiterspace} 2} ]} \frac{\left(-w\right)^{k} \left(\frac{1}{2} \right)_{(n-1)-k} \left(2z\right)^{\left[(n-1)-2k\right]} }{k!\left[(n-1)-k)\right]!}\]

\[\; \; \; \; \; =\left(\frac{n}{zw} \right)P_{n} (z,w)\; -\left(\frac{1}{2z^{2} } \right)P_{n-1} (z,w)\]
\\
Now, on some simplification, we get our desired result.
\end{proof}

\begin{corollary} Following recurrence relation for the Legendre polynomials $P_{n} (z,1)$ holds true
\begin{eqnarray} \label{GrindEQ__3_3_}
&&\frac{1}{w} \frac{\partial }{\partial z} P_{n} (z,1)\; -\left(\frac{n}{z} \right)P_{n} (z,1)\; +\left(\frac{1}{2z^{2} } \right)P_{n-1} (z,1)=0\\
&&\hskip20mm \left(\Re(\alpha) >-1, \; \Re(\beta) >-1, |z|<1 \right)\nonumber.
\end{eqnarray}
\end{corollary}

\begin{proof}
First, substitute w=1 in equation \eqref{GrindEQ__4_2_} we get:
\begin{equation} \label{GrindEQ__4_4_}
P_{n} (z,1)=\sum _{k=0}^{[{n\mathord{\left/ {\vphantom {n 2}}
\right. \kern-\nulldelimiterspace} 2} ]} \frac{\left(-1\right)^{k}
\left(\frac{1}{2} \right)_{n-k} \left(2z\right)^{n-2k} }{k!(n-k)!}.
\end{equation}

Taking differential operator $\Delta _{z} =\left(\frac{1}{w} \frac{\partial }{\partial z} \right)$ then following the same process used in the above theorem leads to the desired result.
\end{proof}

\begin{corollary} Following recurrence relation for the Legendre polynomials $P_{n} (1,w)$ holds true
\begin{eqnarray} \label{GrindEQ__3_4_}
&&\frac{1}{z} \frac{\partial }{\partial w} P_{n} (1,w)\; -\left(\frac{n}{w} \right)P_{n} (1,w)\; +\left(\frac{1}{2} \right)P_{n-1} (1,w)=0\\
&&\hskip20mm \left(\Re(\alpha) >-1, \; \Re(\beta) >-1, |w|<1 \right)\nonumber.
\end{eqnarray}
\end{corollary}

\begin{proof}
 Put z=1 in equation \eqref{GrindEQ__4_2_} we get:
\begin{equation} \label{GrindEQ__4_5_}
P_{n} (1,w)=\sum _{k=0}^{[{n\mathord{\left/ {\vphantom {n 2}}
\right. \kern-\nulldelimiterspace} 2} ]} \frac{\left(-w\right)^{k}
\left(\frac{1}{2} \right)_{n-k} \left(2\right)^{n-2k} }{k!(n-k)!} .
\end{equation}

Taking differential operator $\Delta _{w} =\left(\frac{1}{z} \frac{\partial }{\partial w} \right)$ then following the same process used in the above theorem leads to the desired result.
\end{proof}

\end{document}